\documentclass[a4paper,12pt,reqno]{amsart}
\usepackage{amssymb,amsfonts,amscd,latexsym}
\usepackage{enumerate}
\usepackage{amsmath, mathrsfs,units}
\usepackage{xy}
\xyoption{matrix} \xyoption{arrow}

\numberwithin{equation}{section}

\newcommand{\ad}[1]{\textup{ad}\,{#1}}
\newcommand{\ol}{\overline}
\newcommand{\Ol}[2]{\overline{{#2}}^{\mkern2mu{#1}}}
\newcommand{\rest}[2]{{{#1}_{\kern-.5pt|{#2}}}}
\newcommand{\RR}{\mathbb{R}}
\newcommand{\NN}{\mathbb{N}}

\newcommand{\CC}{\mathbb{C}}
\newcommand{\mloc}{M_{\textup{\rm loc}}(A)}
\newcommand{\Mloc}[1]{M_{\textup{\rm loc}}({#1})}
\newcommand{\Mlocbig}[1]{M_{\textup{\rm loc}}\bigl({#1}\bigr)}
\newcommand{\Mlocit}[2]{{M_{\textup{\rm loc}}^{{(#2)}}({#1})}}
\newcommand{\eps}{\varepsilon}
\newcommand{\pfi}{\varphi}

\newcommand{\prim}[1]{\textup{\rm Prim}({#1})}
\newcommand{\prima}{{\textup{\rm Prim}(A)}}

\newcommand{\sepa}{{\textup{\rm Sep}(A)}}
\newcommand{\maxza}{{\textup{\rm Max}(Z(A))}}
\newcommand{\Ice}[1]{{{\mathscr I}_{\mkern-3mu\textup{\rm ce}}({#1})}}
\newcommand{\longrightarrowraised}{{\hbox{\raise.5\jot\hbox{\scriptsize$\longrightarrow$}}}}
\newcommand{\longleftarrowraised}{{\hbox{\raise.5\jot\hbox{\scriptsize$\longleftarrow$}}}}
\newcommand{\dirlim}{{\smash{\underset{\longrightarrowraised}
                   {\operatorname{lim}}}\vphantom{a^{}_{a_f^{}}}}}

\newcommand{\Dirlim}[1]{{\smash{\dirlim{}}\sp{}_{\,{#1}}
                   \vphantom{a^{}_{a_f^{}}}}}
\newcommand\alglim{{\smash{\underset{\longrightarrowraised}
                   {\operatorname{alg\,lim}}}\vphantom{a^{}_{a_f^{}}}}}
\newcommand\Alglim[1]{{\smash{\alglim{}}\sp{}_{\,{#1}}
                   \vphantom{a^{}_{a_f^{}}}}}

\newcommand{\clos}[1]{{\kern.07em{}^{\textup{c}}\kern-.11em{#1}}}          
\newcommand{\closa}{{\clos{\kern-.05em{A}}}}                               
\newcommand{\supp}[1]{{\textup{c}({#1})}}         

\newcommand{\calD}{{\mathcal D}}
\newcommand{\calT}{{\mathcal T}}

\newcommand{\sh}[1]{{\mathfrak{#1}}}

\newcommand{\shma}{{\sh M_A}}

\newcommand{\shia}{{\sh I_A}}

\newcommand{\shc}{{\sh C}}

\newcommand{\bund}[1]{{\mathsf{#1}}}
\newcommand{\bundA}{{\bund A}}
\newcommand{\bundB}{{\bund B}}
\newcommand{\bundI}{{\bund I}}

\def\C*{{\sl C*}-algebra}
\def\Cs*{{\sl C*}-subalgebra}
\def\Cbund*{{\sl C*}-bundle}
\def\Calg*{{\sl C*}-algebraic}
\def\AF/{{\sl A\kern-.5pt F}-algebra}
\def\AW*{{\sl AW*}-algebra}
\def\CXsheaf/{{$\shc(X)$-sheaf}}
\def\CXsheaves/{{$\shc(X)$-sheaves}}

\newtheorem{lem}{Lemma}[section]
\newtheorem{corol}[lem]{Corollary}
\newtheorem{theor}[lem]{Theorem}
\newtheorem{prop}[lem]{Proposition}

\theoremstyle{remark}

\newtheorem{exem}[lem]{\bf Example}

\newtheorem{remark}[lem]{Remark}

\author{Pere Ara}
\address{Departament de Matem\`atiques\\
Universitat Aut\`onoma de Barcelona\\
08193 Bellaterra (Barcelona)\\
Spain}
\email{para@mat.uab.cat}
\author{Martin Mathieu}
\address{Department of Pure Mathematics\\
Queen's University Belfast\\
Belfast BT7 1NN\\
Northern Ireland}
\email{m.m@qub.ac.uk}

\thanks{The first-named author was partially supported by DGI MICIIN-FEDER MTM2008-06201-C02-01
and by the Comissionat per Universitats i Recerca de la Generalitat de Catalunya
through the grant 2009SGR 1389. This work was carried out during a stay of the second-named author
at the Centre de Recerca Matem\`atica (Barcelona) supported by the Ministerio de Educaci\'on under SAB2009-0147.}

\title[The second local multiplier algebra]{When is the second local multiplier algebra of a \textsl{C*}-algebra equal to the first?}

\begin{document}

\begin{abstract}
We discuss necessary as well as sufficient conditions for the second iterated local multiplier algebra of a separable \C* to agree with the first.
\end{abstract}

\maketitle

\section{Introduction}\label{sect:intro}

\noindent
After the first example of a \C* $A$ with the property that the second local multiplier algebra $\Mloc\mloc$ of $A$
differs from its first, $\mloc$, was found in~\cite{AM06}---thus answering a question first raised in~\cite{Ped78}---, the
behaviour of higher local multiplier algebras began to attract some attention; see, e.g., \cite{AM08}, \cite{ArgFar09},~\cite{ArgFarNext}.
That the local multiplier algebra can have a somewhat complicated structure was already exhibited in~\cite{AM99}, where
an example of a non-simple unital \C* $A$ was given such that $\mloc$ is simple (and hence, evidently, $\Mloc\mloc=\mloc$ in this case).

It was proved in \cite{Somerset} that, if $A$ is a separable unital \C*, $\Mloc\mloc=\mloc$, provided the primitive ideal space $\prima$
contains a dense $G_\delta$ subset of closed points. One of our goals here is to see how this result can be obtained in a straightforward
manner using the techniques developed in~\cite{AM10}.
The key to our argument is the following observation. Every element in $\mloc$ can be realised as a bounded continuous section,
defined on a dense $G_\delta$ subset of $\prima$, with values in the upper semicontinuous \Cbund* canonically associated to the multiplier
sheaf of~$A$. The second local multiplier algebra $\Mloc\mloc$ is contained in the injective envelope $I(A)$ of~$A$,
cf.~\cite{FrankPaulsen}, \cite{AM08}, and every element of $I(A)$ has a similar description as a continuous section of a \Cbund*
corresponding to the injective envelope sheaf of~$A$. To show that $\Mloc\mloc\subseteq\mloc$ it thus suffices to relate
sections of these bundles in an appropriate way. In fact, we shall obtain a more general result in Section~\ref{sect:sheaves}
which, in particular, unifies the commutative and the unital case. The notion of a quasicentral \C*, first studied by Delaroche
\cite{Del67, Del68}, turns out to be crucial.

It emerges, however, that the short answer to the long question in this paper's title is: \textit{rarely}.
In Section~\ref{sect:results}, we provide a systematic approach to producing separable \C*s with the property that their second local multiplier algebra
contains the first as a proper \Cs*. We obtain a quick proof of Somerset's result \cite{Somerset} that $\Mlocit A2=\Mlocit A3$
for a separable \C* $A$ which has a dense $G_\delta$ subset of closed points in its primitive spectrum in Theorem~\ref{thm:mloc3=mloc2} below.
In our approach, the injective envelope is employed as a `universe' in which all \C*s considered
are contained as \Cs*s. However, in contrast to previous studies, we do not need additional information on the injective envelope itself.

In the following we will focus on separable \C*s for a variety of reasons. One of them is the non-commutative Tietze extension theorem,
another one the need for a strictly positive element in the bounded central closure of the \C*. Moreover, just as in
Somerset's paper~\cite{Somerset}, Polish spaces (in the primitive spectrum) will play a decisive role.
Sections~\ref{sect:prelims} and~\ref{sect:results} are fairly self-contained, while Section~\ref{sect:sheaves} relies on the sheaf
theory developed in~\cite{AM10}.

\section{Preliminaries}\label{sect:prelims}

\noindent
For a \C*~$A$, we denote by $\prima$ its primitive ideal space (with the Jacobson topology); this is second countable if $A$
is separable. For an open subset $U\subseteq\prima$, let $A(U)$ stand for the closed ideal of $A$ corresponding to~$U$.
(Hence, $t\in U$ if and only if $A(U)\nsubseteq t$.) We denote by $\calD$ and $\calT$ the sets of dense open subsets and
dense $G_\delta$ subsets of $\prima$, respectively, and consider them directed under reverse inclusion. The local multiplier
$\mloc$ is defined by $\mloc=\Dirlim{U\in\calD}\,M(A(U))$, where, for $U,V\in\calD$ with $V\subseteq U$, the injective *-homomorphism
$M(A(U))\to M(A(V))$ is given by restriction. We put $Z=Z(\mloc)$, the centre of $\mloc$.
For more details on, and properties of, $\mloc$, we refer to~\cite{AM03}.

A point $t\in\prima$ is said to be \textit{separated\/} if $t$ and every point $t'\in\prima$ which is not in the closure of~$\{t\}$ can be
separated by disjoint neighbourhoods. Let $\sepa$ be the set of all separated points of a \C*~$A$. If $A$ is separable then $\sepa$
is a dense $G_\delta$ subset of $\prima$ \cite[Th\'eoreme~19]{Dix68}.

The following result is useful when computing the norm of a (local) multiplier.

\begin{lem}\label{lem:cty-on-dense-Gdelta}
Let\/ $A$ be a separable \C*, and let\/ $T\subseteq\sepa$ be a dense $G_\delta$ subset. For a countable family\/ $\{f_n\mid n\in\NN\}$
of bounded lower semicontinuous real-valued functions on\/ $T$ there exists a dense $G_\delta$ subset\/ $T'\subseteq T$ such that\/
$\rest{f_n}{T'}$ is continuous for each\/ $n\in\NN$.
\end{lem}

\noindent
This is an immediate consequence of the following well-known facts:
$\sepa$ is a Polish space (that is, homeomorphic to a separable, complete metric space) by \cite[Corollaire~20]{Dix68}
and hence any $G_\delta$ subset of $\sepa$ is a Polish space \cite[4.2.2]{Ped79};
every Polish space is a Baire space \cite[4.2.5]{Ped79};
any bounded Borel function into $\RR$ defined on a Polish space can be restricted to a continuous function on some dense
$G_\delta$ subset of the domain \cite[Sect.~32.II]{Kura}.

\medskip
In~\cite{Somerset}, p.~322, Somerset introduces an interesting \Cs* of $\mloc$, which we will denote by $K_A$: \
$K_A$ is the closure of the set of all elements of the form $\sum_{n\in\NN}a_nz_n$, where $\{a_n\}\subseteq A$ is a bounded family
and $\{z_n\}\subseteq Z$ consists of mutually orthogonal projections.
(These infinite sums exist in $\mloc$ by \cite[Lemma~3.3.6]{AM03}, for example. Note also that $Z$ is countably decomposable since
$A$ is separable.) It is easy to see that, if the above family $\{a_n\}$
is chosen from $K_A$ instead of~$A$, then the sum $\sum_{n\in\NN}a_nz_n$ still belongs to~$K_A$ (\cite[Lemma~2.5]{Somerset}).

The significance of the \Cs* $K_A$ is explained by the following result.
Let $\Ice A$ denote the set of all closed essential ideals of a \C*~$A$.
We denote by $\Mlocit{A}{n}=\Mlocbig{\Mlocit{A}{n-1}}$, $n\geq2$ the $n$-fold iterated local multiplier algebra of~$A$.

\begin{lem}\label{lem:why-K}
Let\/ $A$ be a \C* such that\/ $K_A\in\Ice\mloc$.
\renewcommand\theenumi{\roman{enumi}}
\renewcommand\labelenumi{\rm(\theenumi)}
\begin{enumerate}
\item If\/ $K_I=K_A$ for all\/ $I\in\Ice A$ then\/ $\Mloc{K_A}=M(K_A)$.
\item If\/ $\Mloc{K_A}=M(K_A)$ then\/ $\Mlocit{A}{n+1}=\Mlocit{A}{n}$ for all\/ $n\geq2$.
\end{enumerate}
\end{lem}
\begin{proof}
Let $J\in\Ice{K_A}$; then $M(K_A)\subseteq M(J)$. Let $I=J\cap A$; then $I\in\Ice A$ by \cite[Lemma~2.3.2]{AM03}.
By assumption, we therefore have $K_I=K_A$. Let $m\in M(J)$. As $mI\subseteq K_A$, whenever $\{x_n\}$ is a bounded
family in $I$ and $\{z_n\}$ is a family of mutually orthogonal projections in $Z$, we obtain
\[
m\bigl(\sum_n x_nz_n\bigr)=\sum_n mx_nz_n\in K_A
\]
entailing that $mK_A=mK_I\subseteq K_A$, that is, $m\in M(K_A)$. Consequently, \hbox{$M(J)\subseteq M(K_A)$} which implies~(i).

Towards (ii) observe that $M(K_A)=\Mloc{K_A}=\Mloc\mloc$ by hypothesis.
Let $J\in\Ice{\Mlocit{A}{2}}$. Then $J\cap K_A\in\Ice\mloc$ and, since $J\in\Ice{M(K_A)}$, $J\cap K_A\in\Ice J$. As a result,
\begin{equation*}
M(J)\subseteq M(J\cap K_A)\subseteq\Mloc\mloc=M(K_A)
\end{equation*}
and the reverse inclusion $M(K_A)\subseteq M(J)$ is obvious. We conclude that $\Mlocit{A}{3}=M(K_A)=\Mlocit{A}{2}$
which entails the result.
\end{proof}

The next result tells us how to detect multipliers of $K_A$ inside~$I(A)$.
\begin{lem}\label{lem:mult-KA}
Let\/ $A$ be a separable \C* and let\/ $y\in I(A)$. If\/ $ya\in K_A$ for all\/ $a\in A$ then\/ $y\in M(K_A)$.
\end{lem}
\begin{proof}
It suffices to show that $y\sum_{n=1}^\infty z_na_n=\sum_{n=1}^\infty z_nya_n$ whenever  $\{a_n\mid n\in\NN\}\subseteq A$ is a bounded family
and $\{z_n\mid n\in\NN\}\subseteq Z$ consists of mutually orthogonal projections, by \cite[Lemma~2.5]{Somerset}.
Without loss of generality we can assume that $\sum_{n=1}^\infty z_n=1$.

Putting $y'=y\sum_{n=1}^\infty z_na_n\in I(A)$ we observe that
\[
z_jy'=yz_j\sum_{n=1}^\infty z_na_n=yz_ja_j=z_jya_j\in K_A
\]
by hypothesis. It is therefore enough to prove that, if $y'\in I(A)$ and $y'z_j\in K_A$ for all $j\in\NN$, where
$\{z_j\mid j\in\NN\}\subseteq Z$ consists of mutually orthogonal projections with $\sum_{j=1}^\infty z_j=1$,
then $y'=\sum_{j=1}^\infty y'z_j$, where the latter is computed in~$K_A$.

The assumption $y'z_j\in K_A$ for all $j\in\NN$ enables us to write $\sum_{j=1}^\infty y'z_j=\sum_{i=1}^\infty w_ia_i$
for some bounded sequence $(a_i)_{i\in\NN}$ in $A$ and a sequence $(w_i)_{i\in\NN}$ consisting of mutually orthogonal
central projections with $\sum_{i=1}^\infty w_i=1$. For each $n\in\NN$,
\[
(w_1+\ldots+w_n)y'=\sum_{i=1}^n w_ia_i.
\]
Each projection $w_i$ comes with a closed ideal $I_i=w_i\mloc\cap A$ and the {\sl C*}-direct sum $I=\bigoplus_{i=1}^\infty I_i$
is a closed essential ideal of~$A$. For $x_i\in I_i$, $1\leq i\leq n$, we have
\begin{equation*}
\begin{split}
\bigl(y'-\sum_{i=1}^\infty w_ia_i\bigr)(x_1+\ldots+x_n)
    &=\bigl(y'-\sum_{i=1}^\infty w_ia_i\bigr)(w_1+\ldots+w_n)(x_1+\ldots+x_n)\\
    &=\Bigl(\sum_{i=1}^n w_ia_i-\sum_{i=1}^n w_ia_i\Bigr)(x_1+\ldots+x_n)=0.
\end{split}
\end{equation*}
Therefore $\bigl(y'-\sum_{i=1}^\infty w_ia_i\bigr)x=0$ for all $x\in I$ which implies that $y'=\sum_{i=1}^\infty w_ia_i$
by \cite[Proposition~2.12]{AM08}.
\end{proof}

Recall that the \textit{bounded central closure}, $\closa$, of a \C* $A$ is the \Cs* $\ol{AZ}$ of $\mloc$ \cite[Section~3.2]{AM03}.
If $A$ is separable then $\closa$ is $\sigma$-unital, which will be useful in Section~\ref{sect:results}.

In Section~\ref{sect:sheaves}, we shall need the following auxiliary result whose proof is analogous to the one of
\cite[Lemma 2.2]{Somerset} but we include it here for completeness.

\begin{lem}\label{lem:like-somersets}
Let\/ $A$ be a separable \C*, $B$ a \Cs* of\/ $\mloc$ containing\/~$A$, and\/ $J$ a closed essential ideal of\/ $B$.
There is\/ $h\in J$ such that\/ $hz\ne0$ for each non-zero projection\/ $z\in Z$.
\end{lem}
\begin{proof}
By \cite[Proposition~2.14]{AM08}, $I(A)=I(B)=I(\mloc)$ and thus $Z({\Mloc B})=Z$ by \cite[Theorem~4.12]{AM08}.
For $x\in\mloc$, let $\supp x$ denote the central support of~$x$, see \cite{AM03}, page~52 and Remark~3.3.3.
Let $\{h_i\}$ be a maximal family of norm-one
elements $h_i\in J$ such that their central supports $\supp{h_i}$ are mutually orthogonal. Since $A$ is separable, $Z$ is countably decomposable,
hence we may enumerate the non-zero central supports as $\supp{h_n}$, $n\in\NN$. Put $h=\sum_{n=1}^\infty 2^{-n}h_n\in J$.
As $J$ is essential, for a non-zero projection $z\in Z$, there is $h'\in J$ with $h'z\ne0$. If $hz=0$ then $\supp h z=0$ and hence $\supp{h_n}z=0$
for all $n\in\NN$. It follows that $\supp{h_n}\supp{h'z}\leq\supp{h_n}z=0$ which would lead to a contradiction to the
maximality assumption on~$\{h_n\}$. As a result, $hz\ne0$ for every non-zero projection $z\in Z$.
\end{proof}

\section{The second local multiplier algebra}\label{sect:results}

\noindent
In this section we discuss some necessary and some sufficient conditions for the first and the second local multiplier
algebra of a separable \C* $A$ to coincide. The general strategy is that this cannot happen if and only if $\mloc$ contains an essential ideal $K$
with the property that $M(K)\setminus\mloc\neq\emptyset$.

The following proposition introduces the decisive topological condition in~$\prima$.

\begin{prop}\label{prop:K-ess_ideal}
Let\/ $A$ be a separable \C* such that\/ $\prima$ contains a dense $G_\delta$ subset consisting of closed points.
Then\/ $K_A$ is an essential ideal in\/ $\mloc$.
\end{prop}
\begin{proof}
Since $K_A$ is a \Cs* of $\mloc$, it suffices to show that, whenever $m$ is a multiplier of a closed essential ideal of $A$
and $a\in K_A$, $ma\in K_A$; in fact, we can assume that $a\in A$, by Lemma~\ref{lem:mult-KA}.

Let $U\subseteq\prima$ be a dense open subset and take $m\in M(A(U))$. For $t\in U$, let $\tilde t\in\prim{M(A(U))}$
denote the corresponding primitive ideal under the canonical identification of $\prima$ with an open dense subset
of $\prim{M(A(U))}$. Let \hbox{$\{b_n\mid n\in\NN\}$} be a countable dense subset of $A$, and let $T$
be the dense $G_\delta$ subset $T=\sepa\cap U$.
Note that, by Lemma~\ref{lem:cty-on-dense-Gdelta}, there is a dense $G_\delta$ subset $T'\subseteq T$ such that
$t\mapsto\|(m-b_n)a+\tilde t\|$ is continuous for all $n\in\NN$ when restricted to~$T'$.

Let $\eps>0$ and take $t\in T'$. Since $A$ is separable and $t$ is a closed point, the canonical mapping
$M(A(U))\to M(A/t)$ is surjective \cite[3.12.10]{Ped79} and, denoting by $\tilde m$ the image of $m$ under this mapping, we have
$(m-b_n)a+\tilde t=(\tilde m-(b_n+t))(a+t)$. As $\{b_n+t\mid n\in\NN\}$ is dense in $A/t$ and $A/t$ is strictly dense
in its multiplier algebra, there is $b_k$ such that $\|(\tilde m-(b_k+t))(a+t)\|<\eps$. By the above-mentioned
continuity there is therefore an open subset $V\subseteq\prima$ containing~$t$ such that
\[
\|(m-b_k)a+\tilde s\|<\eps\qquad(s\in V\cap T').
\]
Letting $z=z_V\in Z$ be the projection from $A(V)+A(V)^\perp$ onto $A(V)$ we conclude that
$\|zma-zb_ka\|=\sup\limits_{s\in V\cap T'}\|(m-b_k)a+\tilde s\|\leq\eps$.

We now choose a (necessarily countable) maximal family $\{z_k\}\subseteq Z$ of mutually orthogonal projections such that
$\|z_kma-z_kb_ka\|\leq\eps$ for each~$k$. Then $\sup z_k=1$ and
$\bigl\|\sum_k (z_kma-z_kb_ka)\bigr\|\leq\eps$. As $ma=\sum_k z_kma$ and $\sum_k z_kb_ka\in K_A$ we conclude that
$ma\in K_A$ as claimed proving that $K_A$ is an ideal in $\mloc$.

In order to show that $K_A$ is essential let $y\in\mloc$ be such that $yK_A=0$. Then, in particular, $yA=0$ and thus $y=0$
by \cite[Proposition~2.3.3]{AM03}.
\end{proof}

The next result was first obtained in \cite[Theorem~2.7]{Somerset} but we believe our approach is more direct and more conceptual.

\begin{theor}\label{thm:mloc3=mloc2}
Let\/ $A$ be a separable \C* such that\/ $\prima$ contains a dense $G_\delta$ subset consisting of closed points.
Then\/ $\Mlocit{A}{3}=\Mlocit{A}{2}$ and coincides with~$M(K_A)$.
\end{theor}
\begin{proof}
Combining Proposition~\ref{prop:K-ess_ideal} with Lemma~\ref{lem:why-K} all we need to show is that $K_I=K_A$ for each $I\in\Ice A$.
Taking $I\in\Ice A$, the inclusion $K_I\subseteq K_A$ is evident. Let $U\subseteq\prima$ be the open dense subset such that $I=A(U)$.
Let $T\subseteq\prima$ be a dense $G_\delta$ subset consisting of closed and separated points. Fix $a\in A$ and let $\eps>0$.
For $t\in U\cap T$, $(I+t)/t=A/t$ as $t$ is a closed point. Therefore there is $y\in I$ such that $y+t=a+t$ and hence $N(a-y)(t)=0$.
The continuity of the norm function at $t$ (\cite[Lemma~6.4]{AM10}) yields an open neighbourhood $V$ of $t$ such that
$N(a-y)(s)<\eps$ for all $s\in V$. Letting $z=z_V\in Z$ be the projection corresponding to $V$ we obtain $\|z(a-y)\|\leq\eps$.
The same maximality argument as in the proof of Proposition~\ref{prop:K-ess_ideal} provides us with a family $\{z_k\}$ of
mutually orthogonal projections in $Z$ and a bounded family $\{y_k\}$ in $I$ with the property that
$\bigl\|a-\sum_ky_kz_k\bigr\|\leq\eps$. This shows that $A\subseteq K_I$ and as a result $K_A\subseteq K_I$ as claimed.
\end{proof}

It was shown in~\cite{ArgFar09}, see also \cite[Section~6]{AM08}, that the \C* $A=C[0,1]\otimes K(H)$, where $H=\ell^2$, has the property
that $\mloc\neq\Mloc\mloc$. In the following result, we explore a sufficient condition on the primitive ideal space that guarantees
this phenomenon to happen.

We shall make use of some topological concepts. Recall that a topological space $X$ is called \textit{perfect\/}
if it does not contain any isolated points. If the closure of each open subset of $X$ is open, then $X$ is said to be
\textit{extremally disconnected}. Thus, $X$ is not extremally disconnected if and only if it contains an open subset
which has non-empty boundary. It is a known fact that an extremally disconnected metric space must be discrete.

\begin{theor}\label{thm:prima-bad-enough}
Let\/ $X$ be a perfect, second countable, locally compact Hausdorff space.
Let\/ $A=C_0(X)\otimes B$ for some non-unital  separable simple \C*\/ $B$. Then\/ $\mloc\neq\Mloc\mloc$.
\end{theor}
\begin{proof}
Since every point in $\prima=X$ is closed and separated, $K_A$ is an essential ideal in $\mloc$, by Proposition~\ref{prop:K-ess_ideal}.
By Theorem~\ref{thm:mloc3=mloc2}, $\Mloc\mloc=M(K_A)$. To prove the statement of the theorem it thus suffices to find an element
in $M(K_A)$ not contained in~$\mloc$.

Note that every non-empty open subset $O\subseteq X$ contains an open subset which has
non-empty boundary. This follows from the above-mentioned fact and the assumption that $O$ is second countable,
locally compact Hausdorff and hence metrisable. Therefore, if $O$ were extremally disconnected, it had to be discrete
in contradiction to the hypothesis that $X$ is perfect.

Let $\{V_n'\mid n\in\NN\}$ be a countable basis for the topology of~$X$.
For each $n\in\NN$, choose an open subset $V_n$ of $X$ such that $\ol{V_n}\subseteq V_n'$ and $\ol{V_n}$ is not open.
Put $W_n=X\setminus\ol{V_{n}}$. Then $O_n=V_{n}\cup W_n$ is a dense open subset of~$X$.

Let $z_n$ denote the equivalence class of $\chi^{}_{V_{n}}\otimes1\in C_b(O_n,M(B)_\beta)=M(C_0(O_n)\otimes B)$ in~$Z$.
Let $(e_n)_{n\in\NN}$ be a strictly increasing approximate identity of $B$ with the properties
$e_ne_{n+1}=e_n$ and $\|e_{n+1}-e_n\|=1$ for all~$n$; see \cite[Lemma~1.2.3]{Loring}, e.g.
Put $p_1=e_1$, $p_n=e_n-e_{n-1}$ for $n\geq2$. Then $(p_{2n})_{n\in\NN}$ is a sequence of mutually orthogonal positive norm-one
elements in~$B$.
Set $q_n=\sum_{j=1}^n z_jp_{2j}$, $n\in\NN$, where we identify an element $b\in M(B)$ canonically with the constant function
in $M(A)=C_b(X,M(B)_\beta)$. By means of this we obtain an increasing sequence $(q_n)_{n\in\NN}$ of positive elements in $\mloc$
bounded by~$1$.
Since the injective envelope is monotone complete~\cite{Ham79}, the supremum of this sequence exists in $I(A)$ and is a positive element of norm~$1$,
which we will write as $q=\sup_n q_n=\sum_{n=1}^\infty z_np_{2n}$.

Suppose that $q\in\mloc$. Then, for given $0<\eps<1/2$, there are a dense open subset $U\subseteq X$ and
$m\in C_b(U,M(B)_\beta)_+$ with $\|m\|\leq1$ such that $\|m-q\|<\eps$.
Upon multiplying both on the left and on the right by~$p_{2n}^{\nicefrac12}$ we find that
\[
\sup_{t\in U\cap O_n}\!\!\bigl\|p_{2n}^{\nicefrac12}m(t)p_{2n}^{\nicefrac12}-\chi_{V_n}(t)p_{2n}^2\bigr\|
         =\bigl\|p_{2n}^{\nicefrac12}mp_{2n}^{\nicefrac12}-z_np_{2n}^2\bigr\|<\eps.
\]
Let $n\in\NN$ be such that $V_n'\subseteq U$.
Define $f_n\in C_b(U)$ by $f_n(t)=\|p_{2n}^{\nicefrac12}m(t)p_{2n}^{\nicefrac12}\|$, $t\in U$
(note that $p_{2n}^{\nicefrac12}mp_{2n}^{\nicefrac12}\in C_b(U,B)$). Then $0\leq f_{n}\leq1$ and
\begin{equation*}
\begin{split}
\bigl|f_n(t)-\chi_{V_n}(t)\bigr|
  &= \bigl| \|p_{2n}^{\nicefrac12}m(t)p_{2n}^{\nicefrac12}\| - \chi_{V_n}(t)\|p_{2n}^2\| \bigr|\\
  &\leq \bigl\|p_{2n}^{\nicefrac12}m(t)p_{2n}^{\nicefrac12}-\chi_{V_n}(t)p_{2n}^2\bigr\|<\eps\\
\end{split}
\end{equation*}
for all $t\in U\cap O_n$.
By construction, $\ol{V_n}$ is not open; hence $\partial \ol{V_n}\ne\emptyset$. Each $t_0\in\partial \ol{V_n}$ also
belongs to $\ol{W_n\cap V_n'}$ as $\partial\ol{V_n}=\partial W_n$ and
hence $t_0\in \ol{W_n}\cap V_n'\subseteq\ol{W_n\cap V_n'}$ since $V_n'$ is open.
For every $t\in V_n$, $|f_{n}(t)-1|<\eps$ and hence $f_{n}(t)\geq1-\eps>1/2$ for all $t\in\ol{V_n}$, by continuity of~$f_{n}$.
In particular, $f_{n}(t_0)>1/2$.
For every $t\in W_n\cap V_n'$, we have $f_{n}(t)<\eps<1/2$ and thus $f_{n}(t_0)\leq\eps<1/2$.
This contradiction shows that $q\notin\mloc$.

\smallskip
In order to prove that $q$ belongs to $M(K_A)$  it suffices to show that $qa\in K_A$ for every $a\in A$, by Lemma~\ref{lem:mult-KA}.
For each $n\in\NN$ and $a\in A$, $q_na\in\closa$ since \hbox{$z_jp_{2j}a\in ZA$}. Therefore,
$q_n\in M(\closa)$ for each~$n$. Note that $\closa$ contains a strictly positive element~$h$.
Indeed, taking an increasing approximate identity $(g_n)_{n\in\NN}$ of $C_0(X)$ we obtain an increasing approximate
identity $u_n=g_n\otimes e_n$, $n\in\NN$ of~$A$. It follows easily that $(u_n)_{n\in\NN}$ is an approximate identity for
$\closa=\ol{AZ} $.  It is well-known that $h=\sum_{n=1}^\infty 2^{-n}u_n$ is then a strictly positive element.

As a result, in order to prove that $(q_n)_{n\in\NN}$ is a Cauchy sequence in $M(\closa)_\beta$,
we only need to show that $(q_nh)_{n\in\NN}$ is a Cauchy sequence. For $k\in\NN$, $p_{2j}e_k=(e_{2j}-e_{2j-1})e_k=0$
if $2j>k+1$. Consequently,
\[
z_jp_{2j}h=\sum_{k=1}^\infty2^{-k}z_jp_{2j}u_k=\sum_{k=1}^\infty2^{-k}g_kz_jp_{2j}e_k\qquad(j\in\NN)
\]
yields that, for each $n\in\NN$,
\begin{equation*}
\begin{split}
q_nh &=\sum_{j=1}^n\sum_{k=1}^\infty2^{-k}g_kz_jp_{2j}e_k\\
            &=\sum_{k=1}^\infty2^{-k}g_kz_1p_2e_k + \sum_{k=3}^\infty2^{-k}g_kz_2p_4e_k+\ldots+\!\sum_{k=2n-1}^\infty2^{-k}g_kz_np_{2n}e_k.
\end{split}
\end{equation*}
We conclude that, for $m>n$,
\[
\|(q_m-q_n)h\|=\Bigl\|\sum_{j=n+1}^m\sum_{k=2j-1}^\infty2^{-k} g_kz_jp_{2j}e_k\Bigr\|
                              =\max_{n+1\leq j\leq m}\,\Bigl\|\sum_{k=2j-1}^\infty2^{-k} g_kz_jp_{2j}e_k\Bigr\|\leq\!\sum_{k=2n+1}^\infty2^{-k}
\]
since $g_kz_jp_{2j}e_k\,g_\ell z_ip_{2i}e_\ell=0$ for all $k,\ell$ whenever $i\ne j$;
therefore $\|(q_m-q_n)h\|\to0$ as $n\to\infty$. This proves that $(q_n)_{n\in\NN}$ is a strict Cauchy sequence in~$M(\closa)$.
Let $\tilde q\in M(\closa)$ denotes its limit, which is a positive element of norm at most one since $M(\closa)_+$ is closed
in the strict topology. In order to show that $\tilde q=q$ note at first that $I(M(\closa))=I(\closa)=I(A)$
by \cite[Proposition~2.14]{AM08}. The mutual orthogonality of the $p_{2n}\!$'s yields
$qq_n=q_mq_n$ for all $m\geq n$. Thus, for all $a\in\closa$, $aqq_n=aq_mq_n$ which implies that $aqq_n=a\tilde qq_n$
for all~$a$. As $A$ is essential in $I(A)$, it follows that $qq_n=\tilde qq_n$ for all $n\in\NN$ by \cite[Theorem~3.4]{AM08}.
Repeating the same argument using the strict convergence of $(q_n)_{n\in\NN}$ we obtain that $q\tilde q={\tilde q}^{\,2}$.

For all $1\leq n\leq m$, $q_n\leq q_m$ and hence $a^*q_na\leq a^*q_ma$ for every $a\in\closa$. It follows that, for all~$n$,
$a^*q_na\leq a^*\tilde qa$ for every~$a$ and therefore $q_n\leq\tilde q$ for all $n\in\NN$. Consequently, $q\leq\tilde q$.
Together with the above identity $(\tilde q-q)\tilde q=0$ this entails that $q=\tilde q\in M(\closa)$.

Finally, for each $a\in A$, we have $qa\in\closa\subseteq K_A$. This completes the proof.
\end{proof}
\begin{remark}\label{rem:stonean}
A space $X$ as in Theorem~\ref{thm:prima-bad-enough} is perfect if and only if it contains a dense $G_\delta$ subset with empty interior.
In \cite[Theorem~6.13]{AM08}, the existence of a dense $G_\delta$ subset with empty interior in the primitive spectrum, which was
assumed to be Stonean, was used to obtain a \C* $A$ such that $\mloc$ is a proper subalgebra of~$I(A)$ and the latter agreed
with $\Mloc\mloc$. In contrast to this example, and also the one considered in~\cite{ArgFar09}, our approach in
Theorem~\ref{thm:prima-bad-enough} does not need any additional information on the injective envelope; nevertheless all
higher local multiplier algebras coincide by Theorem~\ref{thm:mloc3=mloc2}.
\end{remark}
\begin{remark}\label{rem:dichotomy}
Taking the two results Corollary~\ref{cor:central-mloc} and Theorem~\ref{thm:prima-bad-enough} together we obtain the following,
maybe surprising dichotomy for a compact Hausdorff space $X$ satisfying the assumptions in $\eqref{thm:prima-bad-enough}$.
Let $A=C(X)\otimes B$ for a unital, separable, simple \C*~$B$. Then $\mloc=\Mloc\mloc$.
But if we stabilise $A$ to $A_s=A\otimes K(H)$ then $\Mloc{A_s}\neq\Mloc{\Mloc{A_s}}$!
\end{remark}

With a little more effort we can replace the commutative \C* in Theorem~\ref{thm:prima-bad-enough} by a nuclear one, provided the properties
of the primitive ideal space are preserved. We shall formulate this as a necessary condition on a \C* $A$ with tensor product structure to
enjoy the property $\Mloc\mloc=\mloc$. Note that, whenever $B$ and $C$ are separable \C*s and at least one of them is nuclear,
the primitive ideal space $\prim{C\otimes B}$ is homeomorphic to $\prim C\times\prim B$, by \cite[Theorem~B.45]{RaeWil}, for example.

Some elementary observations are collected in the next lemma in order not to obscure the proof of the main result.
\begin{lem}\label{lem:elem-topo}
Let\/ $X$ be a topological space, and let\/ $G\subseteq X$ be a dense  subset consisting of closed points.
\begin{enumerate}[\rm(i)]
\item If\/ $X$ is perfect then\/ $G$ is perfect (in itself).
\item For each\/ $V\subseteq X$ open, $\ol V\cap G=\Ol{G}{V\cap G}$, where
            $\Ol{G}{\phantom{x}}$ denotes the closure relative to~$G$.
\item For each\/ $V\subseteq X$ open, $\partial\bigl(\Ol{G}{V\cap G}\bigr)=\partial\ol V\cap G$.
\end{enumerate}
\end{lem}
\begin{proof}
Assertion (i) is immediate from the density of $G$ and the hypothesis that $X\setminus\{t\}$ is open for each $t\in G$.
In~(ii), the inclusion ``$\supseteq$'' is evident. The other inclusion ``$\subseteq$'' follows from the density of~$G$.

To verify (iii), we conclude from (ii) that
\begin{equation*}
G\setminus\Ol{G}{V\cap G} =G\setminus(\ol V\cap G)=G\cap(X\setminus\ol V)
\end{equation*}
and therefore, with $W=X\setminus\ol V$,
\begin{equation*}
\Ol{G}{G\setminus\Ol{G}{V\cap G}}=\Ol{G}{G\cap W}=G\cap\ol W,
\end{equation*}
where we used~(ii) another time. This entails
\begin{equation*}
\partial\bigl(\Ol{G}{V\cap G}\bigr)=\Ol{G}{V\cap G} \cap \Ol{G}{G\setminus\Ol{G}{V\cap G}}=\ol V\cap G\cap\ol{X\setminus\ol V}=\partial\ol V\cap G
\end{equation*}
as claimed.
\end{proof}

\begin{theor}\label{thm:nec-condition}
Let\/ $B$ and\/ $C$ be separable \C*s and suppose that at least one of them is nuclear. Suppose further that\/ $B$ is simple and non-unital
and that\/ $\prim C$ contains a dense $G_\delta$ subset consisting of closed points.
Let\/ $A=C\otimes B$. If\/ $\mloc=\Mloc\mloc$ then\/ $\prim C$ contains an isolated point.
\end{theor}
\begin{proof}
Let $X=\prim C=\prima$. We shall assume that $X$  is perfect
and conclude from this that $\Mloc\mloc\neq\mloc$. By Proposition~\ref{prop:K-ess_ideal}, $K_A$ is an essential ideal in $\mloc$.
Using Theorem~\ref{thm:mloc3=mloc2} we are left with the task to find an element in $M(K_A)\setminus\mloc$.

The hypothesis on $X$ combined with the separability assumption yields a dense $G_\delta$ subset $S\subseteq X$
consisting of closed separated points which is a Polish space. By Lemma~\ref{lem:elem-topo} (i), $S$ is a perfect metrisable space and therefore
cannot be extremally disconnected, as mentioned before. Since a non-empty open subset of a perfect space is clearly perfect, it follows that
every non-empty  open subset of $S$ contains an open subset which has non-empty boundary.

Let $\{V_n'\mid n\in\NN\}$ be a countable basis for the topology of~$X$.
For each $n\in\NN$, choose an open subset $V_n$ of $X$ such that $\Ol{S}{V_n\cap S}\subseteq V_n'\cap S$
and $\Ol{S}{V_n\cap S}$ is not open.
By Lemma~\ref{lem:elem-topo}~(ii), $\Ol{S}{V_n\cap S}=\ol{V_n}\cap S$ and we shall use the latter, simpler notation.
Put $W_n=X\setminus\ol{V_{n}}$. Then $O_n=V_{n}\cup W_n$ is a dense open subset of~$X$.

Using the same notation as in the fourth paragraph of the proof of Theorem~\ref{thm:prima-bad-enough} we define the element $q\in I(A)$
by $q=\sum_{n=1}^\infty z_n\otimes p_{2n}$. The argument showing that $q\in M(K_A)$ takes over verbatim from
the proof of Theorem~\ref{thm:prima-bad-enough}.  We will now modify the argument in the fifth paragraph of that proof.

Suppose that $q\in\mloc$. For $0<\eps<1/4$, there are a dense open subset $U\subseteq X$ and an element $m\in M(A(U))_+$ with
$\|m\|\leq1$ such that $\|m-q\|<\eps$.
Let $n\in\NN$ be such that $V_n'\subseteq U$ and choose $t_0\in\partial\ol{V_n}\cap S\subseteq U\cap S$ using
Lemma~\ref{lem:elem-topo}~(iii). Since the ideal $C(U)$ of $C$ corresponding to $U$ is not contained in~$t_0$,
there is $c\in C(U)_+$ with $\|c\|=1$ and $\|c+t_0\|=1$. As the function $t\mapsto\|c+t\|$ is lower semicontinuous, there is an open
subset $V\subseteq U$ containing $t_0$ such that $\|c+t\|>1-\eps$ for $t\in V$. Let
$a=c^{\nicefrac12}\otimes{p_{2n}^{\nicefrac12}}\in C(U)\otimes B=A(U)$ and put $f(t)=\|ama+t\|$, $t\in U$. By \cite[Lemma~6.4]{AM10},
$f$ is continuous on $U\cap S$ because $ama\in A$. For each $t\in V\cap O_n$ we have
\begin{equation*}
\begin{split}
\bigl|f(t)-\chi_{V_n}(t)\bigr| &\leq \bigl|\|ama+t\|-\chi_{V_n}(t)\,\|c+t\|\bigr| + \bigl|\chi_{V_n}(t)\,\|c+t\|-\chi_{V_n}(t)\bigr|\\
                                                         &\leq \bigl\|ama+t-\chi_{V_n}(t)c\otimes p_{2n}^2+t\bigr\| + (1-\|c+t\|)\,\chi_{V_n}(t)\\
                                                         &\leq \bigl\|(c^{\nicefrac12}\otimes{p_{2n}^{\nicefrac12}})\,m\,(c^{\nicefrac12}\otimes{p_{2n}^{\nicefrac12}})
                                                                                 -cz_n\otimes p_{2n}^2\bigr\| + \eps\\
                                                         &\leq\|m-q\|+\eps<2\,\eps,
\end{split}
\end{equation*}
since $(c^{\nicefrac12}\otimes{p_{2n}^{\nicefrac12}})\,q\,(c^{\nicefrac12}\otimes{p_{2n}^{\nicefrac12}})=cz_n\otimes p_{2n}^2$.
For each $t\in V_n\cap S$ we have $f(t)>1-2\eps>1/2$ and therefore $f(t_0)>1/2$ by continuity
of $f$ on $U\cap S$ and the fact that $\Ol{S}{V_n\cap S}=\ol{V_n}\cap S$ by Lemma~\ref{lem:elem-topo}~(ii),
thus $t_0\in\Ol{S}{V\cap V_n\cap S}$. On the other hand,
\[
t_0\in\ol{W_n}\cap V\cap S\subseteq\ol{W_n\cap V}\cap S =\Ol{S}{W_n\cap V\cap S}
\]
as $V$ is open and using Lemma~\ref{lem:elem-topo}~(ii) again. Thus $f_{n}(t_0)\leq2\,\eps<1/2$.
This contradiction shows that $q\notin\mloc$, and the proof is complete.
\end{proof}

We can now formulate an if-and-only-if condition characterising when the second local multiplier algebra is equal to the first.

\begin{corol}\label{cor:if-and-only-if}
Let\/ $A=C\otimes B$ for two separable \C*s\/ $B$ and\/ $C$ satisfying the conditions of Theorem~\ref{thm:nec-condition}.
Suppose that\/ $\prima$ contains a dense $G_\delta$ subset consisting of closed points.
Then\/ $\mloc=\Mloc\mloc$ if and only if\/ $\prima$ contains a dense subset of isolated points.
\end{corol}
\begin{proof}
Let $X=\prima$, $X_1$ the set of isolated points in $X$ and $X_2=X\setminus\ol{X_1}$. Then $X_1$ and $X_2$ are open subsets of~$X$
with corresponding closed ideals $I_1=A(X_1)$ and $I_2=A(X_2)$ of~$A$.
If $X_1$ is dense, $I_1$ is the minimal essential closed ideal of $A$ so $\mloc=M(I_1)$. It follows that
\[
\Mloc\mloc=\Mloc{M(I_1)}=\Mloc{I_1}=\mloc.
\]
In the general case, $\mloc=\Mloc{I_1}\oplus\Mloc{I_2}$ by \cite[Lemmas 3.3.4 and 3.3.6]{AM03}.
If $X_2\ne\emptyset$, it contains a dense $G_\delta$ subset of closed points and so $I_2=C(X_2)\otimes B$ satisfies all
the assumptions in Theorem~\ref{thm:nec-condition} while $X_2$ is a perfect space. It follows that
\begin{equation*}
\begin{split}
\Mloc\mloc &= \Mloc{\Mloc{I_1}\oplus\Mloc{I_2}}=\Mloc{\Mloc{I_1}}\oplus\Mloc{\Mloc{I_2}}\\
                        &\ne\Mloc{I_1}\oplus\Mloc{I_2}=\mloc.\qedhere
\end{split}
\end{equation*}
\end{proof}

\section{A sheaf-theoretic approach}\label{sect:sheaves}

\noindent
In~\cite{AM10}, we developed the basics of a sheaf theory for general \C*s. This enabled us to establish the following formula
for $\mloc$ in \cite[Theorem~7.6]{AM10}:
\[
\mloc=\Alglim{T\in\mathcal T}\,\Gamma_b(T,\bundA),
\]
where
$\bundA$ is the upper semicontinuous \Cbund* canonically associated to the multiplier sheaf $\shma$ of $A$ \cite[Theorem~5.6]{AM10} and
$\Gamma_b(T,\bundA)$ is the \C* of bounded continuous sections of $\bundA$ on~$T$. A like description is available for the
injective envelope:
\[
I(A)=\Alglim{T\in\mathcal T}\,\Gamma_b(T,\bundI),
\]
where the \Cbund* $\bundI$ corresponds to the injective envelope sheaf $\shia$ of~$A$, see \cite[Theorem~7.7]{AM10}.
These descriptions are compatible, by \cite[Corollary~7.8]{AM10}.
Since a continuous section is determined by its restriction to a dense subset,
the *-homomorphisms $\Gamma_b(T,\bundB)\to\Gamma_b(T',\bundB)$,
$T'\subseteq T$, $T'\in\mathcal T$ are injective for any \Cbund* $\bundB$ and thus isometric. Consequently,
an element $y\in\Mloc\mloc$ is contained in some \Cs* $\Gamma_b(T,\bundI)$ and will belong to $\mloc$ once we find
$T'\subseteq T$, $T'\in\mathcal T$ such that $y\in\Gamma_b(T',\bundA)$.

\begin{remark}\label{rem:usc-to-cts}
Let $a\in\Gamma_b(T,\bundA)$ for a separable \C*~$A$. By applying Lemma~\ref{lem:cty-on-dense-Gdelta} to the negative of
the upper semicontinuous norm function on~$\bundA$, there is always a smaller dense $G_\delta$
subset $S\subseteq\sepa\cap T$ on which the restriction of the function $t\mapsto\|a(t)\|$ is continuous.
\end{remark}

On the basis of this, we shall obtain a concise proof of an extension of one of Somerset's main results, \cite[Theorem~2.7]{Somerset},
in this section. This extension is twofold: firstly, we replace the assumption of an identity by the more general hypothesis on~$A$
to be quasicentral. Secondly, we establish the result for \Cs*s of $\mloc$ containing~$A$.

The following concept was introduced and initially studied by Delaroche~\cite{Del67, Del68}.
A \C* $A$ is called \textit{quasicentral\/} if no primitive ideal of $A$ contains the centre $Z(A)$ of~$A$.
We recall some basic properties of quasicentral \C*s.

\begin{remark}\label{rem:quasicentral-basic}
Let $A$ be a quasicentral \C*.
\begin{enumerate}[(i)]
\item The mapping $\nu\colon\prima\to\maxza$, $t\mapsto t\cap Z(A)$ is well-defined, surjective and continuous.
\item The Dauns--Hofmann isomorphism $Z(M(A))\to C_b(\prima)$, $z\mapsto f_z$ such that
            $za+t=f_z(t)(a+t)$ for all $a\in A$, $z\in Z(M(A))$ and $t\in\prima$ maps $Z(A)$ onto $C_0(\prima)$;
            see~\cite[A.34]{RaeWil} and~\cite[Proposition~1]{Del67}.
\item Every approximate identity of $Z(A)$ is an approximate identity for $A$ and thus $A=Z(A)A$;
            see~\cite[Proposition~1]{Arch75}.
\end{enumerate}
\end{remark}

Part~(i) of the result below on the existence of local identities is already contained in \cite[Proposition~2]{Del67}
but we provide an independent brief proof along the lines of the proof of \cite[Theorem~5]{Arch75}.

\begin{lem}\label{lem:local-identities}
Let\/ $A$ be a quasicentral \C*, $C\subseteq\prima$\/ compact and\/ $t\in C$.
\begin{enumerate}[\rm(i)]
\item There exists\/ $z\in Z(A)_+$, $\|z\|=1$ such that\/ $z+s=1_{A/s}$, the identity in the primitive quotient\/ $A/s$ for all\/ $s\in C$.
\item Let\/ $U_1$ be an open neighbourhood of\/ $t$ contained in\/ $C$ and let\/ $U_2=\prima\setminus\ol{U_1}$. 
           If\/ $z\in Z(A)_+$ is as in~{\rm(i)} then\/ $z+A(U_2)$ is the identity in\/ $A/A(U_2)$.
\end{enumerate}
\end{lem}
\begin{proof}
As $\maxza$ is a locally compact Hausdorff space, there is $f\in C_0(\maxza)_+$ with $\|f\|=1$ such that
$f(\nu(s))=1$ for all $s\in C$ \cite[1.7.5]{Ped89}. Identifying $Z(A)$ with $C_0(\prima)$, see Remark~\ref{rem:quasicentral-basic} above,
we obtain $z\in Z(A)_+$, $\|z\|=1$ such that $f_z=f\circ\nu$ and hence $z+s=1_{A/s}$ for all $s\in C$. This proves~(i).

Now let $U_1$ be an open neighbourhood of $t$ contained in $C$ and put $U_2=\prima\setminus\ol{U_1}$. Let $z\in Z(A)_+$ be as in~(i).
Then $\ol{U_1}=\{s\in\prima\mid A(U_2)\subseteq s\}$ is homeomorphic to $\prim{A/A(U_2)}$ via
$s\mapsto s/A(U_2)$ \;\cite[4.1.11]{Ped79}.
Therefore, any identity which holds in $\bigl(A/A(U_2)\bigr)\big/\bigl(s/A(U_2)\bigr)$ for a dense set of~$s$ holds in
$A/A(U_2)$. Since $\bigl(A/A(U_2)\bigr)\big/\bigl(s/A(U_2)\bigr)\cong A/s$ and $z+s=1_{A/s}$ for all $s\in U_1$,
it follows that $z+A(U_2)=1_{A/A(U_2)}$ as claimed in~(ii).
\end{proof}

With the help of Lemma~\ref{lem:local-identities} we can extend a key result in \cite{AM10}, viz.\ \cite[Lemma~6.9]{AM10},
from the unital case to the situation of quasicentral \C*s. Though the proof is similar, we include the details for completeness.

\begin{prop}\label{prop:iso-of-fibres}
Let\/ $A$ be a quasicentral \C*, and let\/ $t\in\prima$ be a closed and separated point. Then the natural mapping\/
$\pfi_t\colon\bundA_t\to A/t$ is an isomorphism.
\end{prop}
\begin{proof}
Since $A$ is quasicentral, the \C* $A/t$ is unital, and since $t$ is a closed point, $A/t$ is simple.
Therefore the natural mapping $\pfi_t\colon\bundA_t\to \Mloc{A/t}$ given by \cite[Proposition~6.2]{AM10}
simplifies to $\pfi_t\colon\bundA_t\to A/t$.
As $t$ is a separated point, $\ker\iota_t=t$ where $\iota_t\colon A\to\bundA_t$ is the canonical map
\cite[Proposition~6.5]{AM10}. Since $\pfi_t\circ\iota_t=\pi_t$, where $\pi_t$ is the canonical surjection $A\to A/t$, we find that $\pfi_t$
is injective when restricted to $\iota_t(A)$.

Let $U$ be an open neighbourhood of $t$ in $\prima$, and take $m\in M(A(U))$.
Let $C$ be a compact neighbourhood of $t$ contained in~$U$ \cite[4.4.4]{Ped79}.
By Lemma~\ref{lem:local-identities}, there is $z\in Z(A)_+$, $\|z\|=1$ such that
$z+s=1_{A/s}$ for all $s\in C$.
Choose $e\in A(U)_+$ with the property that $\|e\|=1$ and $e+t=z+t$ in~$A/t$. (Note that $A(U)+t/t=A/t$ as $A/t$ is simple.)
Since $N(z-e)(t)=0$ and $N(z-e)$ is continuous at~$t$, as $t$ is a separated point \cite[Lemma~6.4]{AM10},
there is an open neighbourhood $U_1$ of $t$ contained in $C$ such
that $N(z-e)(s)<1/2$ for every $s\in U_1$. Set $Y=\ol{U_1}$ and $U_2=\prima\setminus Y$.
By Lemma~\ref{lem:local-identities}~(ii), $z+A(U_2)$ is the identity of $A/A(U_2)$.
Since $A(U_1)$ sits as an essential ideal in $A/A(U_2)$, we have an
embedding of unital \C*s $A/A(U_2)\subseteq M(A(U_1))=\shma(U_1)$. The set
$\{ s\in \prima \mid N(z-e)(s)\le 1/2\}$ is closed in $\prima$ and contains $U_1$; consequently $N(z-e)(s)\le 1/2$ for every $s\in Y$.

Since $N_{A/A(U_2)}\bigl((z-e)+A(U_2)\bigr)(s)=N_A(z-e)(s)\le 1/2$ for every
$s\in Y$, we get that $\|1_{A/A(U_2)}-e+A(U_2)\|=\|(z-e)+A(U_2)\|\le 1/2<1$, and thus
$e+A(U_2)$ is invertible in $A/A(U_2)$. Take any $y\in A$ such that $y+A(U_2)=(e+A(U_2))^{-1}$. Then we have
\begin{equation*}
\rest{m}{\mathfrak M _A(U_1)}
  =\rest{m}{\mathfrak M_A(U_1)}(e+A(U_2))(y+A(U_2))
  =(me +A(U_2))(y+A(U_2))\in A/A(U_2),
\end{equation*}
since $me\in A(U)\subseteq A$. As a result, $\rest{m}{\mathfrak M _A(U_1)}$ belongs to the image of
the map $A\to \shma(U_1)$. We thus find that the image of $m$
in $\bundA_t=\varinjlim _{t\in W} \shma(W)$ belongs to the image
of the map $A\to \bundA_t$, and it turns out that the map $A/t\to \bundA_t$ is
surjective. Since it is also injective, we conclude that it is an isomorphism, and so its inverse, $\varphi _t$, must be an isomorphism too.
\end{proof}

The following example shows that the statement of Proposition~\ref{prop:iso-of-fibres} can fail if the \C* is not quasicentral.

\begin{exem}\label{exam:not_quasicentral}
Let $B=C_b(\NN, M_2(\CC))$ be the \C* of all bounded (continuous) functions from $\NN$ to the $2\times2$ complex matrices.
We shall write elements of $B$ as $x=(x(n))_{n\in\NN}$. Let $A$ be the \Cs* of $B$ consisting of those $x$ such that
$x_{ij}(n)\to0$, $n\to\infty$ for $(i,j)\ne(1,1)$ and $x_{11}(n)\to\mu(x)$, $n\to\infty$.
Then $A$ is a non-unital separable $2$-subhomogeneous \C* with Hausdorff primitive spectrum.
In fact, the primitive ideals of $A$ are given by $t_\infty=\ker\mu$ and, for each $n\in\NN$,
$t_n=\{x\in A\mid x(n)=0\}$ (with corresponding irreducible representations given by
$\pi_\infty\colon A\to\CC$, $\pi_\infty(x)=\mu(x)$ and $\pi_n\colon A\to M_2(\CC)$, $\pi_n(x)=x(n)$, $x\in A$).
Clearly $\prima$ is homeomorphic to the one-point compactification $\NN_\infty$ of~$\NN$,
since $\{U_n\mid n\in\NN\}$ with $A(U_n)=\bigcap_{j=1}^n t_j$ forms a neighbourhood basis for~$t_\infty$.

As $C_0(\NN,M_2(\CC))=t_\infty$, $t_\infty$ is an essential ideal of $A$ and $\mloc=\Mloc{t_\infty}=M(t_\infty)=B=I(A)$.
Moreover, $M(A)$ consists of those $x$ satisfying
$\lim_n x_{12}(n)=\lim_n x_{21}(n)=0$, $\lim_n x_{11}(n)=\mu(x)$ and $(x_{22}(n))_{n\in\NN}$ is bounded.
It follows that $\prim{M(A)}=\beta\NN\cup\{t_\infty\}$, where all the ultrafilters in $\beta\NN$ yield characters of $M(A)$
via $\lim_{\mathcal U}x_{22}(n)$. Any open neighbourhood of $t_\infty$ in $\prim{M(A)}$ must contain one of the $U_n$'s
and hence $t_m$ for $m\geq n+1$. As $\NN$ is dense in $\beta\NN$ we conclude that no point in $\beta\NN\setminus\NN$ can be
separated from~$t_\infty$.

This leads to the following description of the associated upper semicontinuous \Cbund*.
For each $n\in\NN$, $\bundA_{t_n}\cong A/t_n=M_2(\CC)$.
On the other hand, $\bundA_{t_\infty}=\Dirlim{n} M(A(U_n))$
with the connecting mappings given by
\[
\bigl(0,\ldots,0,y(n+1),y(n+2),\ldots\bigr)\longmapsto\bigl(0,\ldots,0,0,y(n+2),\ldots\bigr)
\]
taking into account that $M(A(U_n))\cong M(A)$ for each~$n$.
It follows that $\bundA_{t_\infty}$ is indeed commutative and isomorphic to
$C(\{t_\infty\}\cup\beta\NN\setminus\NN)=\CC\times\ell^\infty/c_0$. As a result, the homomorphism
$\pfi_{t_\infty}\colon\bundA_{t_\infty}\to A/t_\infty=\CC$   is far from being injective.
Note that $t_\infty\supseteq Z(A)\cong c_0$ so that $A$ is not quasicentral.

To complete the picture we note that, in the \Cbund* $\bundI$ associated to the injective envelope sheaf, the fibres are
$\bundI_{t_n}=M_2(\CC)$, $n\in\NN$ and $\bundI_{t_\infty}=M_2(\ell^\infty/c_0)$ with the embedding $\bundA_{t_\infty}\to\bundI_{t_\infty}$
simply the diagonal map.
\end{exem}

A quasicentral \C* $A$ is said to be \textit{central\/} if the mapping $\nu$ of Remark~\ref{rem:quasicentral-basic}~(i)
is injective. Since this is equivalent to the hypothesis that $A$ has Hausdorff primitive spectrum \cite[Proposition~3]{Del67},
the same arguments as in Theorem~6.10 and Corollary~6.11 of \cite{AM10} yield the following consequence.

\begin{corol}\label{cor: central-case}
Let\/ $A$ be a central separable \C*. Then all the fibres\/ $A_t=A/t$, $t\in\prima$ are isomorphic to the fibres\/ $\bundA_t$
associated to the multiplier sheaf\/ $\shma$ of\/~$A$. Indeed, the multiplier sheaf\/ $\shma$ of\/ $A$ is isomorphic
to the sheaf\/ $\Gamma_b(-,\bundA)$ of bounded continuous local sections of the \Cbund*\/ $\bundA$ associated to\/~$\shma$.
\end{corol}

Every \C* $A$ contains a largest quasicentral ideal $J_A$, which is the intersection of all closed ideals in $A$ that contain~$Z(A)$
\cite[Proposition~1]{Del68}. Clearly, the hypothesis in our main result of this section below is equivalent to the assumption
that $J_A$ is essential.

\begin{theor}\label{thm:mlocmloc-is-mloc-one}
Let\/ $A$ be a separable \C*  such that\/ $\prima$ contains a dense $G_\delta$ subset consisting of closed points.
Suppose\/ $A$ contains a quasicentral essential closed ideal.
If\/ $B$ is a \Cs* of\/ $\mloc$ containing\/ $A$ then\/ $\Mloc B\subseteq\mloc$. In particular, $\Mloc\mloc=\mloc$.
\end{theor}
\begin{proof}
As $\Mloc I=\mloc$ for every $I\in\Ice A$, we can assume without loss of generality that $A$ itself is quasicentral.

Take $y\in M(J)$ for some $J\in\Ice B$, and let $T\in\mathcal T$ be such that $y\in\Gamma_b(T,\bundI)$
(recall that $\Mloc B\subseteq I(B)=I(A)$).
By hypothesis, and the fact that $\sepa$ itself is a dense $G_\delta$ subset, we can assume that $T$ consists of closed separated points
of $\prima$.  Take $h\in J$ with the property that $hz\ne0$ for every non-zero projection $z\in Z$ (Lemma~\ref{lem:like-somersets}).
By Remark~\ref{rem:usc-to-cts}, there is $S\in\calT$ contained in $T$ such that the function $t\mapsto\|h(t)\|$ is continuous when
restricted to~$S$ (viewing $h$ as a section in $\Gamma_b(S,\bundA)$). Consequently, the set $S'=\{t\in S\mid h(t)\neq0\}$ is open in $S$
and intersects every $U\in\calD$ non-trivially; it is thus a dense $G_\delta$ subset of $\prima$.
Replacing $T$ by $S'$ if necessary, we may assume that $h(t)\ne0$ for all $t\in T$.

A standard argument yields a separable \Cs* $B'$ of $J$ containing $AhA$ and such that $yB'\subseteq B'$ and $B'y\subseteq B'$.
Let $\{b_n\mid n\in\NN\}$ be a countable dense subset of~$B'$. For each~$n$, let $T_n\in\mathcal T$ be such that
$b_n\in\Gamma_b(T_n,\bundA)$. Letting $T'=\bigcap_n T_n\cap T\in\mathcal T$ we find that $B'\subseteq\Gamma_b(T',\bundA)$ and hence
$B'_t=\{b(t)\mid b\in B'\}\subseteq\bundA_t$ for each $t\in T'$.

For each $t\in T$, the \C*s $\bundA_t$ and $A/t$ are isomorphic, by Proposition~\ref{prop:iso-of-fibres} above, and since $A/t$ is unital and simple
(as $t$ is closed), we obtain $\bundA_t\,h(t)\,\bundA_t=\bundA_t$ for each $t\in T'$.
Consequently,
\[
\bundA_t=\bundA_t\,h(t)\,\bundA_t=(A/t)h(t)(A/t)=A_th(t)A_t=(AhA)_t\subseteq B'_t
\]
and thus $B'_t=\bundA_t$ for all $t\in T'$.
We can therefore find, for each $t\in T'$, an element $b_t\in B'$ such that $b_t(t)=1(t)$.
It follows that $y(t)=y(t)\,1(t)=(yb_t)(t)\in\bundA_t$ for all $t\in T'$, which yields $y\in\Gamma_b(T',\bundA)$.
This proves that $y\in\mloc$.
\end{proof}

\begin{corol}\label{cor:central-mloc}
For every central separable \C*\/ $A$, $\Mloc\mloc=\mloc$.
\end{corol}

In~\cite{Ped78}, Pedersen showed that every derivation of a separable \C* $A$ becomes inner in $\mloc$ when extended
to the local multiplier algebra. His question whether every derivation of $\mloc$ is inner (when $A$ is separable) has since
been open and seems to be connected to the problem how much bigger $\Mloc\mloc$ can be. In this direction, Somerset
proved the next result in \cite[Theorem~2.7]{Somerset} in the unital case. Our approach shows that it is an immediate consequence of Pedersen's
theorem, in view of Theorem~\ref{thm:mlocmloc-is-mloc-one} above.

\begin{corol}\label{cor:der-inner-in-mloc}
Let\/ $A$ be a quasicentral separable \C*  such that\/ $\prima$ contains a dense $G_\delta$ subset consisting of closed points.
Then every derivation of\/ $\mloc$ is inner.
\end{corol}
\begin{proof}
Let $d\colon\mloc\to\mloc$ be a derivation. Let $B$ be a separable \Cs* of $\mloc$ containing $A$ which is invariant under~$d$.
By \cite[Theorem~4.1.11]{AM03}, $d_B=\rest dB$ can be uniquely extended to a derivation $d_{\Mloc B}\colon\Mloc B\to\Mloc B$.
Both derivations can be uniquely extended to their respective injective envelopes, by \cite[Theorem~2.1]{HamOkaSaito},
but since $I(B)=I(\Mloc B)$, we have $d_{I(B)}=d_{I(\Mloc B)}$. The same argument applies to the extension of $d$, since
$I(B)=I(A)=I(\mloc)$; in other words, $d_{I(\mloc)}=d_{I(B)}$ which we will abbreviate to~$\tilde d$.
By \cite[Proposition~2]{Ped78}, $d_{\Mloc B}=\ad y$ for some $y\in\Mloc B$; in fact, $y\in\mloc$ by
Theorem~\ref{thm:mlocmloc-is-mloc-one}. By uniqueness, $\tilde d=\ad y$ and hence $d=\ad y$ on $\mloc$.
\end{proof}

\smallskip

\end{document}